\documentclass[onesite, 12pt, a4paper]{article}
\usepackage{amsmath,amssymb,amsthm}
\usepackage{makeidx}
\usepackage{amsopn,amsgen,amsbsy,amstext,amsfonts}
\usepackage[unicode]{hyperref}
\makeindex
\usepackage{amsmath,amssymb}
\def\ri{\rightarrow}	
																			
							\def\om{\omega}						\def\Om{\Omega}				
													
				\def\ti\Phi{\tilde{\Phi}}			\def\ri{\rightarrow}				
					\def\si{\sigma}										
														
															\def\cB{{\mathcal B}}
														\def\cF{{\mathcal F}}
											
																	\def\om{\omega}
\def\Om{\Omega}												
													\def\ti\Phi{\tilde{\Phi}}
\def\it{\italic}										\def\ovl{\overline}

				\def\PP{{\mathbb P}}									\def\RR{{\mathbb R}}

\newtheorem{theorem}{\bf Theorem}[section]
\newtheorem{definition}[theorem]{\bf Definition}
\newtheorem{corollary}[theorem]{\bf Corollary}
\newtheorem{lemma}[theorem]{\bf Lemma}
\newtheorem{proposition}[theorem]{\bf Proposition}

\newtheorem{example}[theorem]{\bf Example}

\font\it=cmti10

\begin{document}
\makeatletter	   
\renewcommand{\ps@plain}{%
     \renewcommand{\@oddhead}{\textrm{\it The existence of random solutions to random optimization problems}\hskip 3 cm \textrm{T. N. Anh}}%
    \renewcommand{\@evenfoot}{\@oddfoot}			}
\makeatother     
\title{The existence of random solutions to random optimization problems}
\author{Ta Ngoc Anh$^1$}
\date{\today}   
\maketitle
$^1${\it Department of Mathematics,  Le Quy Don technical University (LQDTU), No 236 Hoang Quoc Viet road, Cau Giay district, Hanoi, Vietnam.}  Email: {\it tangocanh@gmail.com}

\setcounter{page}{1}
\begin{abstract} 
	In this paper we study random optimization problems where random functions are investigated in sample paths. Some sufficient conditions ensuring the existence of random solutions to random optimization problems are proposed. 
\end{abstract}
\pagestyle{plain}

\section{Introduction and preliminaries} 		      
Let $f:\RR^n\to \RR$ be a real function. In many situations, we have to consider the problem $\underset{x\in C}{Min} f( x)$, where $C$ is a subset of $\RR^n$ defined by some specified constraints. In fact, by the action of many random factors, the function $f$ and the set $C$ are not determined clearly. It is suitable to consider a random function $f: \Omega\times\RR^n\to \RR$ and random set $C(\omega)$ instead of deterministic function $f$ and fixed set $C$, successively. This leads to consider the problem of finding minima of a random function on a random set, that is called random optimization problem. In recent years, random optimization problems have received increasing attention from both optimization and probability communities. There are some approaches to this problem depending on the ways random functions are examined. Some authors consider the problem  $\underset{x\in C(\omega)}{Min} f(\omega, x)$ where $f(\omega, x)$ is a random function, $C(\omega)$ is a random set. In this problem, random function $f(\omega,x)$ is investigated in sample paths (see \cite{AEL}, \cite{GOR}, \cite{GR}, \cite{SW}, \cite{TY}, \cite{V2}). Some ones examine stochastic programming problem  $\underset{x\in C}{Min} E[f(\omega, x)]$ where $E[f (\omega,x)]$ is the expectation of random function $f(\omega,x)$. Here, random function $f(\omega,x)$ is considered in expected values (see \cite{CF},\cite{LL}, \cite{S}). In two-stage stochastic programming problems, random functions could be investigated both in sample paths and in mean values depending on different stages (see \cite{AHP}, \cite{SDR} and references therein). 
\newpage

In this paper, we consider the problems of finding minima of  random functions on random sets in metric spaces. We will investigate random functions in sample paths. By some aspects, our approach leads to consider a family of deterministic optimization problems indexed by sample space. An important problem in view point of probability is to consider the existence of measurable minimizers for random functions. 
We give some results about measurability of optimal values  and prove the existence of measurable minimizers for some random functions. Especially, we give the sufficient conditions for the existence of measurable minimizers of twice continuously differential random functions defined on finite dimensional Euclidean space $\RR^n$. A bit difference from the deterministic case, the positive definiteness of Hessian matrix of second partial derivatives evaluated at a stationary point doesn't ensure stationary point to become a random solution to local random optimization problem of a twice continuously differential random function. 

Let  $(\Omega, \mathcal F, \PP)$ be a probability space and $X$ be a completely separable metric space (Polish space).  We denote by $\cB(X) $ the Borel $\sigma$-algebra of $X$,  by $2^X$ the family of all non empty subsets of $X$,  by $C(X)$ the family of all non empty closed subsets of $X$ and by $K(X)$ the family of all compact subsets of $X$.  The product $\si-$algebra on $\Omega\times X$  is denoted by $\cF\otimes\cB(X)$. \par
A mapping $\xi: \Omega \to X$ is called an  $X$-valued random variable if   $\xi^{-1}(B)=\{\omega\in \Omega |\xi(\omega)\in B\}\in \cF$ for any $B\in \cB(X)$.  A set-valued mapping  $C:\Omega\to 2^X$  is said to be measurable  if $C^{-1}(B)=\{\omega\in \Omega | C(\omega)\cap B\not=\emptyset\}\in \cF$ for each open subset $B$ of $X.$ The graph of $C$ is defined by $Gr(C)=\{(\omega, x)|\omega\in \Omega, x\in C(\omega)\}.$  A function $\xi:\Omega\to X$ will be said to be a selection of $C:\Omega\to 2^X$ if $\xi(\omega)\in C(\omega)$ for every $\omega\in\Omega$. The following theorem shows the existence of measurable selections of a measurable set-valued mapping. 

\begin{theorem}[\cite{CV}, Theorem III.8]\label{III.8CV}
Let $(\Omega, \mathcal F, \PP)$ be a probability space, $X$ be a Polish space and $C:\Omega\to C(X)$ be a measurable mapping. Then there exists a sequence of measurable selections $\{\xi_n\}$ of $C$ such that $C(\omega)=\overline{\{\xi_n(\omega)\}}$ for all $\omega\in\Omega$. 
\end{theorem}

\begin{theorem} [\cite{HM}, Theorem 3.5]\label{3.5HM}
Let $(\Omega, \mathcal F, \PP)$ be a complete probability space, $X$ be a Polish space and $C:\Omega\to C(X)$. Then $C$ is a measurable mapping if and only if $Gr(C)\in \cF\otimes\cB(X)$.
\end{theorem}

We recall the concept of  random function. 
\begin{definition}
Let  $(\Omega, \mathcal F, \PP)$ be a probability space and $X$ be a separable metric space.
\begin{enumerate}
\item A mapping $f:\Omega\times X\to \RR$  is called a random function on $X$ if for each $x\in X$, the mapping $f(.,x) $ is a real-valued random variable. 
\item  The random function $f:\Omega\times X\to \RR$ is said to be measurable if $f(\omega, x)$  is  jointly measurable w.r.t. $\omega$ and $x$, that is $f^{-1}(B)=\{(\omega,x)\in \Omega\times X|f(\omega,x)\in B\}\in \cF\otimes\cB(X)$ for any $B\in \cB(\RR)$.
\item The random function $f:\Omega\times X\to \RR$ is said to be continuous if for each $\omega$  the function $ f(\omega,.)$  is continuous.
\end{enumerate}
\end{definition}
As the following theorem shows that a continuous random function is also a measurable one but converse is not true in general.
\begin{theorem}[\cite{HM},Theorem 6.1]\label{6.1HM}
Let $(\Omega, \mathcal F, \PP)$ be a probability space, $X$ be a separable metric space, $Y$ be a metric space, and let $f: \Omega\times X\to Y$ be measurable in $\omega$  and continuous in $x$. Then $f$ is measurable. 
\end{theorem}
\begin{theorem}[\cite{HM},Theorem 6.4]\label{6.4HM}
Let $(\Omega, \mathcal F, \PP)$ be a complete probability space, $X$ be a Polish space, $Y$ be a metric space, and let $f: \Omega\times X\to Y$ be measurable in $\omega$ and continuous in $x$. Then, for any closed subset $B$ of $Y$, $\omega\mapsto\{x\in X|f(\omega,x)\in B\}$ defines a measurable mapping from $\Omega$ to $2^X$. 
\end{theorem}

Let  $f: \Omega\times X\to \RR$ be a random function and $\eta$ be a real-valued random variable. An equation of the form 
\begin{equation}\label{eq:1}
f(\omega, x)=\eta(\omega)
\end{equation}
 is called a random equation. We say that the equation \eqref{eq:1} has a deterministic solution if there exists a mapping $\xi:\Omega \to X$  such that $f(\omega,\xi(\omega))=\eta(\omega)$ for each $\omega\in\Omega.$ If in addition $\xi$ is an $X-$valued random variable then we call $\xi$ a random solution of equation  \eqref{eq:1}. In general, a random equation having deterministic solutions may not have a random solution. The following theorem will be used to prove the existence of random solutions to random optimization problems. \par
 
\begin{theorem}[ \cite{TA}, Theorem 2.3]\label{2.3TA}
Let $(\Omega, \mathcal F, \PP)$ be a complete probability space, $X$ be a Polish space. Suppose that $f:\Omega\times X \rightarrow \RR$  is a measurable random function and $\eta:\Omega\to \RR$ is a random variable. Then, the random equation $f(\omega, x)=\eta(\omega)$ has a random solution if and only if it has a deterministic one.
\end{theorem}

\section{Random optimization problems}

Let $(\Omega, \mathcal F, \PP)$ be a complete probability space, $X$ be a metric space and $f: \Omega\times X\to \RR$ be a random function. The random optimization problem denoted by $ROP(f,X)$ is stated as follows: Find a mapping $\xi: \Omega\to X$ such that 
\begin{equation}\label{rop}
\underset{x\in X}{Inf} f(\omega, x)=f(\omega,\xi(\omega))\ \ \  \forall \omega\in \Omega.
\end{equation}
The mapping $\xi$ satisfying \eqref{rop} is called a deterministic solution to $ROP(f,X)$. If in addition $\xi$ is a random variable, then it is called a random solution to  $ROP(f,X)$. Clearly,  if  a random optimization problem has a random solution then it also has a deterministic one. Naturally, some the following questions arise: 
\begin{enumerate}
\item[(Q1)] If a random optimization problem has at least one deterministic solution, does this imply the existence of a random solution?  
\item[(Q2)] Is a unique deterministic solution also a random one?
\item[(Q3)] In what conditions the answer to (Q1) affirmative?
\end{enumerate}
The following example shows that the answers to questions (Q1) and (Q2) are negative. The answer to the question (Q3) is given in the Theorem \ref{ms}. 
\begin{example}\label{rs} 
Let $\Om=[0;1]$ and $\cF$ be the family of subsets $A\subset \Omega$ with the property that either $A$ is countable or the complement $A^c$ is countable.  Define a probability measure $\PP$ on $\cF$ by
\begin{equation*} \PP(A)=\begin{cases} 0&\mbox{if $A$ is  countable}\\
1&\mbox{otherwise.}\end{cases}\end{equation*}
It is clear that $(\Omega,\cF, \PP)$ is a complete probability space. Define random function $f:\Omega\times \RR\to \RR$ by
 \begin{equation*} 
f(\omega,x)= \begin{cases}\frac{x^2}{x^2+1}&\mbox{if $\om=x$ }\\
1&\mbox{if $\omega\not=x$.}\end{cases}
\end{equation*}
\end{example}
It is easy to verify that $f$ is random function and for each $\omega\in\Omega$, $\underset{x\in \RR}{Inf} f(\omega, x)= \dfrac{\omega^2}{\omega^2+1}$
 as $x=\omega$. Thus $\xi(\omega)=\omega$ is a unique deterministic solution to $ROP(f,\RR)$. For  $B=[0;1/2)\in\cB(\RR)$, we have $\xi^{-1}(B)=[0;1/2)\notin\cF$. Thus, $\xi$ is not a random variable and $ROP(f,\RR)$ has not a random solution. \par

Now we investigate the measurability of the optimal values of random functions. Define the mapping $\eta:\Omega\to\RR$ by $\eta(\omega)=\underset{x\in \RR}{Inf} f(\omega, x)$ for all $\omega\in\Omega$.  We have $\eta(\omega)=\dfrac{\omega^2}{\omega^2+1}\forall\omega\in\Omega$. However, $\eta$ is not a random variable. Indeed, we have \[\eta^{-1}\big(0;\frac15\big)=\big\{\omega|\dfrac{\omega^2}{\omega^2+1}<\frac15\big\}=\big[0;\frac12\big)\notin \cF.\]
The following theorem gives a sufficient condition on $f$ ensuring the measurability of minima of a random function.

\begin{theorem}\label{mm} Let $(\Omega, \mathcal F, \PP)$ be a complete probability space and $X$ be a Polish space.  Suppose that $f: \Omega\times X\to \RR$ is a measurable random function and $C:\Omega\to C(X)$ is a measurable mapping. Then the mapping $\eta:\Omega\to \RR$ defined by
\begin{equation}\label{rop1}
\eta(\omega)=\underset{x\in C(\omega)}{Inf} f(\omega, x)\ \ \  \forall \omega\in \Omega
\end{equation}
is a random variable.
\end{theorem}

\begin{proof}
We need to show that the set $\{\omega| \eta(\omega)<t\}$ belongs to $\cF$ for any $t\in \RR$. Indeed, for each $t\in \RR,$
$$\{\omega| \eta(\omega)<t\}=Proj_{\Omega}\big(Gr(C)\cap\{(\omega,x)|f(\omega, x)<t\}\big),$$
where $Proj_A(B)$ denotes the projection of $B$ to $A$.  By the measurability of random function $f$, the set $\{(\omega,x)|f(\omega, x)<t\}$ belongs to $ \cF\otimes\cB(\RR).$ From Theorem \ref{3.5HM} and the measurability of $C$, $Gr(C)$ belongs to  $ \cF\otimes\cB(\RR).$ So, we have $Gr(C)\cap\{(\omega,x)|f(\omega, x)<t\}\in \cF\otimes\cB(\RR).$ By projection theorem due to Castaing and Valadier (\cite{CV}, Theorem III.23) we get $$Proj_{\Omega}\big(Gr(C)\cap\{(\omega,x)|f(\omega, x)<t\}\big)\in \cF.$$ The proof is complete.
\end{proof}
In partcular, when $C(\omega)=X\ \forall\omega\in\Omega$, we have the following corollary.
\begin{corollary}\label{mmc1}
Let $(\Omega, \mathcal F, \PP)$ be a complete probability space, $X$ be a Polish space and $f: \Omega\times X\to \RR$ be a measurable random function. Then the mapping $\eta:\Omega\to \RR$ defined by $\eta(\omega)=\underset{x\in X}{Inf} f(\omega, x)\ \forall \omega\in \Omega$ is a random variable. 
\end{corollary}
The following corollary is implied directly from Theorem \ref{mm} and the fact that a continuous function defined on a compact set attains minimum value. 
\begin{corollary}\label{mmc2}
Let $(\Omega, \mathcal F, \PP)$ be a complete probability space, $X$ be a Polish space. Suppose that $C:\Omega\to K(X)$ is a measurable mapping and $f: \Omega\times X\to \RR$ is a continuous random function. Then the mapping $\eta:\Omega\to \RR$ defined by $\eta(\omega)=\underset{x\in C(\omega)}{Min} f(\omega, x)\ \forall \omega\in \Omega$ is a random variable. 
\end{corollary}
Noting that the measurability of random function in Theorem \ref{mm} is only sufficient not necessary condition as showed in  the following example.

\begin{example}\label{nm}
Let $(\Omega, \cF, \PP)$ be a complete probability space where $\Omega=\{0;1\}, \cF=\{\emptyset, \Omega\}$ and $D$ be a non-Borel subset of $\RR$. We define the mapping $f:\Omega\times\RR\ri\RR$ as follows 
$$f(0, x)=f(1,x)=\begin{cases}
0&if \ \ \  x\in D \\ 
1& if \ \ \ x\in \ovl D\\ 
\end{cases}$$
where $\ovl D= \RR\setminus D.$ 
\end{example}
For each $x\in \RR$, we have $f(\omega, x)=0\ \forall\omega$ as $x\in D$ and $f(\omega, x)=1\ \forall\omega$ as $x\in \ovl D$. Thus, $f$ is a random function. For any $\omega\in\Omega, \eta(\omega)=\underset{x\in \RR}{Inf} f(\omega, x)=0$ so $\eta$ is a random variable. However, $f$ is not a measurable random function. Indeed, we have $B=\{0\}\in \cB(\RR)$ and $f^{-1}(B)=\Omega\times D\notin \cF\otimes\cB(\RR)$.

\begin{theorem}\label{ms} Let $(\Omega, \mathcal F, \PP)$ be a complete probability space and $X$ be a Polish space.  Suppose that $f: \Omega\times X\to \RR$ is a measurable random function and $C:\Omega\to C(X)$ is a measurable mapping. If random optimization problem 
\begin{equation}\label{op}
 \underset{x\in C(\omega)}{Inf} f(\omega, x)
\end{equation}
has a deterministic solution then it also has a random solution.
\end{theorem}
\begin{proof}
For each $\omega\in\Omega$, denote $ \underset{x\in C(\omega)}{Inf} f(\omega, x)$ by $\eta(\omega)$. By Theorem \ref{mm}, $\eta$ is a random variable. The random solutions to the problem \eqref{op} are also the random solutions of random equation $f(\omega, x)=\eta(\omega)$. The existence of deterministic solution to random optimization problem  $ \underset{x\in C(\omega)}{Inf} f(\omega, x)$ implies that the random equation $f(\omega, x)=\eta(\omega)$ has a deterministic solution.  By Theorem \ref{2.3TA}, the proof is completed.
\end{proof}
It is the fact that a continuous function defined on a compact set in metric space attains minimum value. So we have the following result.
\begin{corollary}\label{mmcf}
Let $(\Omega, \mathcal F, \PP)$ be a complete probability space, $X$ be a Polish space, $C:\Omega\to K(X)$ be a measurable mapping and $f: \Omega\times X\to \RR$ be a continuous random function. Then the random optimization problem $\underset{x\in C(\omega)}{Min} f(\omega, x)$ has a random solution.  
\end{corollary}
Noting that the measurability of random function in Theorem \ref{ms} is not necessary condition to ensure the existence of a random solution. Indeed, let $f: \Omega\times \RR\to \RR$ be a random function defined in Example \ref{nm}. Then, $f$ is non-measurable random function. However, $\xi(\omega)=c\in D$ for any $\omega\in\Omega$ is a random solution to random optimization problem $\underset{x\in \RR}{Inf} f(\omega, x).$\par
We now consider random local optimization problems.  
\begin{definition}\label{lrop}
Let $(\Omega, \mathcal F, \PP)$ be a probability space, $X$ be a separable metric space and  $f: \Omega\times X\to \RR$ be a random function. We call mapping $\xi:\Omega\to X$ a deterministic solution to random local optimization problem denoted by $RLOP(f, X)$ if, for each $\omega\in \Omega$, there exists a positive real number $\delta(\omega)$ such that 
\begin{equation}\label{rlopeq}
\underset{x\in B(\xi(\omega), \delta(\omega))}{Inf} f(\omega, x)=f(\omega,\xi(\omega)),
\end{equation}
where $B(\xi(\omega), \delta(\omega))$ is an open ball of center $\xi(\omega)$ and radius $\delta(\omega).$ We call mapping $\xi:\Omega\to X$ a random solution if it is a deterministic solution to $RLOP(f, X)$ and is also a random variable.
\end{definition}
Example \ref{rs} shows that not any deterministic solution to a random local optimization problem is a random one. Before giving some conditions ensuring the existence of random solutions to $RLOP(f, X)$, we recall some basic tools.\par
A symmetric matrix $H\in \RR^{n\times n}$ is said to be positive definite if  $v^t H v > 0$ for all nonzero vector $v\in \RR^n$ and is said to be positive semi-definite if  $v^t H v \geq 0$ for all vector $v\in \RR^n$.  Sylvester's criterion says that a symmetric matrix $H$ is positive definite if and only if determinants of all submatrices taken from the top left corner of $H$ are positive. We denote by $H\succ 0$ and $H\succcurlyeq 0$ the positive definite matrix and the  positive semi-definite one, successively.  \par

Let $f: \RR^n\to \RR$ be a multivariate real function. We denote by $g(x)$ and $H(x)$ the gradient and Hessian matrix of second partial derivatives of $f$ evaluated at $x\in \RR^n$, successively. We have $g(x)=(g_1(x),..., g_n(x))$ and $H(x)=(h_{ij}(x))_{n\times n}$, where $g_i(x)=\dfrac{\partial f(x)}{\partial  x_i}, i=1, ..., n$  and $h_{ij}(x)=\dfrac{\partial^2 f(x)}{\partial  x_i\partial  x_j}, i, j=1, ..., n$. If $f$ is twice continuously differentiable then $H(x)$ is a symmetric matrix. \par
We now consider the case of random function defined on finite dimensional Euclidean space $\RR^n.$ From now we suppose that $(\Omega, \mathcal F, \PP)$ is a complete probability space. The following propositions give necessary conditions for a random variable to become random solution to a random local optimization problem.
\begin{proposition}\label{pr1}
Let  $f: \Omega\times \RR^n\to \RR$ be a random function such that $f(\omega,x)$ is continuously differentiable in $x$ for each $\omega\in\Omega$. If  random variable $\xi:\Omega\to \RR^n$ is a random solution to random local optimization problem $RLOP(f, \RR^n)$ then $g(\omega,\xi(\omega))=0$ for each $\omega\in\Omega.$
\end{proposition}
In case of twice continuously differentiable random function, we have.
\begin{proposition}\label{pr2}
Let  $f: \Omega\times \RR^n\to \RR$ be a random function such that $f(\omega,x)$ is twice continuously differentiable in $x$ for each $\omega\in\Omega$. If  random variable $\xi:\Omega\to \RR^n$ is a random solution to random local optimization problem $RLOP(f, \RR^n)$ then $g(\omega,\xi(\omega))=0$ and $H(\omega,\xi(\omega))$ is positive semi-definite for each $\omega\in\Omega$.
\end{proposition}
The proofs of Proposition \ref{pr1} and Proposition \ref{pr2} can be given by some standard arguments as in classical analysis.\par
The following theorem gives a sufficient condition for the existence of a random solution to random local optimization problem. 
\begin{theorem}\label{mslrop}
Let  $f: \Omega\times \RR^n\to \RR$ be a random function such that $f(\omega,x)$ is twice continuously differentiable in $x$ for each $\omega\in\Omega$. If  there exits a mapping $\xi:\Omega\to \RR^n$ such that $g(\omega,\xi(\omega))=0$ and $H(\omega,\xi(\omega))$ is positive definite for each $\omega\in\Omega$ then the random local optimization problem $RLOP(f, \RR^n)$ has a random solution.
\end{theorem}

\begin{proof}
Firstly, we show that there exits a random variable $\psi: \Omega\to \RR^n$ such that $g(\omega,\psi(\omega))=0$ and $H(\omega,\psi(\omega))$ is positive definite for each $\omega\in\Omega$. To do this, we need the following lemma.
\begin{lemma}\label{mi}
Let $(\Omega, \mathcal F, \PP)$ be a complete probability space, $X$ be a Polish space and $C_i:\Omega\to C(X)$ be a measurable mapping for each $i\in N$. Then the mapping $C:\Omega\to C(X)$ defined by $C(\omega)=\underset{i}{\cap}C_i(\omega)$ is a measurable mapping. 
\end{lemma}
\begin{proof}
We will show that $Gr(C)$ is a measurable set. We have 
\begin{align*}
Gr(C)&=\{(\omega,x)|\omega\in\Omega, x\in C(\omega)\}=\{(\omega,x)|\omega\in\Omega, x\in \underset{i}{\cap}C_i(\omega)\}\\
&= \underset{i}{\cap}\{(\omega,x)|\omega\in\Omega, x\in C_i(\omega)\}\\
&= \underset{i}{\cap}Gr(C_i)
\end{align*}
By the measurability of $C_i$ and Theorem \ref{3.5HM}, $Gr(C_i)\in \cF\otimes\cB(X)$ for any $i\in N$. Thus $Gr(C)\in  \cF\otimes\cB(X)$. Once again, by Theorem \ref{3.5HM}, $C$ is a measurable mapping.
\end{proof}
Denote by $g_i(\omega, x)$ the $i$th partial derivative of $f(\omega,x)$, that is $g_i(\omega,x)=\dfrac{\partial f(\omega, x)}{\partial  x_i}, i=1, ..., n$ and $h_{ij}(\omega, x)=\dfrac{\partial^2 f(\omega, x)}{\partial  x_i\partial  x_j}, i, j=1, ..., n$. By the assumption, $g_i(\omega,x)$ and $h_{ij}(\omega, x)$ are continuous random functions for any $i, j=1, ..., n.$\par
Define $C_i:\Omega\to 2^{\RR^n}$ by $C_i(\omega)=\{x\in \RR^n|g_i(\omega,x)=0\}, i=1,..., n; \omega\in\Omega$. By Theorem \ref{6.4HM}, $C_i$ is a measurable mapping for any $i=1,...,n.$ By the continuity of $g_i(\omega, x)$ in $x$, and the existence of mapping  $\xi:\Omega\to \RR^n$ such that $g(\omega,\xi(\omega))=0$ we imply that $C_i(\omega)$ is a non empty closed set for each $\omega\in\Omega$.  
Let $C:\Omega\to \RR^n$ be a mapping defined by $C(\omega)=\underset{i}{\cap}C_i(\omega) =\{x\in\RR^n|g(\omega,x)=0\}$ for any $\omega\in\Omega$. By Lemma \ref{mi}, $C$ is a measurable mapping valued in $C(\RR^n)$.\par 
We denote $\Delta_k(\omega, x)$ the $k$th determinant of submatrix taken from the top left corner of $H(\omega, x), k=1,...,n$.  By the continuity of   $h_{ij}(\omega, x), i, j=1,..., n$, we imply that $(\omega,x)\mapsto\Delta_k(\omega, x)$ is a continuous random function for any $k=1,...,n.$ By Theorem \ref{6.1HM} and the continuity of $\Delta_k(\omega, x)$, it is implied that $\Delta_k(\omega, x)$ is a measurable random function and the set $\{(\omega,x)|\omega\in\Omega, x\in \RR^n, \Delta_k(\omega, x)>0\}$ is measurable for each $k=1,...,n$.
We define $D: \Omega\to 2^{\RR^n}$ by $D(\omega)=\{x\in\RR^n| H(\omega, x)\succ 0\}$. We will show that $D$ is a measurable mapping valued in $C(\RR^n)$. Indeed, because of the continuity in $x$ of all entries $h_{ij}(\omega,x)$ and the fact that $H(\omega, x)$ is positive definite as $x\in D(\omega)$, we imply that $D(\omega)$ is closed set for any $\omega\in\Omega$. By the existence of mapping $\xi:\Omega\to \RR^n$ such that $H(\omega,\xi(\omega))$ is positive definite for each $\omega\in\Omega$, $D(\omega)$ is non empty set.  By Sylvester's criterion, we have 
\begin{align*}
Gr(D)&=\{(\omega,x)|\omega\in\Omega, x\in D(\omega)\}=\{(\omega,x)|\omega\in\Omega, x\in\RR^n, H(\omega,x)\succ 0\}\\
&= \underset{k}{\cap}\{(\omega,x)|\omega\in\Omega, x\in \RR^n, \Delta_k(\omega, x)>0\}.
\end{align*}
Thus, $Gr(D)$ belongs to $\cF\otimes\cB(\RR^n)$. By Theorem \ref{3.5HM}, $D$ is a measurable mapping valued in $C(\RR^n)$.\par 
Set $M:\Omega\to C(\RR^n)$ defined by $M(\omega)=C(\omega)\cap D(\omega)$ for any $\omega\in\Omega$. By Lemma \ref{mi}, $M$ is a measurable mapping. By the existence of mapping $\xi:\Omega\to \RR^n$ in the assumption, $M(\omega)$ is a non empty closed set for any $\omega\in\Omega$. By Theorem\ref{III.8CV}, there exists a measurable selection $\psi:\Omega\to \RR^n$ of $M$, that is $\psi(\omega)\in M(\omega)$ for any $\omega\in\Omega$. For each $\omega\in\Omega$, we have $g(\omega,\psi(\omega))=0$ and $H(\omega, \psi(\omega))$ is a positive definite matrix. \par
Secondly, we will show that $\psi:\Omega\to \RR^n$ is a random solution to random local optimization problem $RLOP(f,\RR^n).$ For each fixed $\omega\in\Omega,$ by Taylor expansion, for any $d\in \RR^n$, we have
\[f(\omega, \psi(\omega) +d)=f(\omega, \psi(\omega))+g(\omega, \psi(\omega)).d +\frac12 d^t.H(\omega, \psi(\omega) +\theta.d).d \]
for some $\theta\in (0,1)$. By the continuity of $H(\omega,x)$ and the positive definiteness of $H(\omega, \psi(\omega))$, there exists $\delta(\omega)>0$ such that for any direction $d$ with $\|d\|<\delta(\omega)$ and any scalar $\theta\in (0,1)$, $ H(\omega, \psi(\omega) + \theta.d)$ is positive definite. Thus,
for any $d\in \RR^n, d\not=0$ such that $\|d\|<\delta(\omega)$,  from above expansion and $g(\omega, \psi(\omega) )=0$ we have
\[f(\omega, \psi(\omega) + d) > f(\omega, \psi(\omega)).\]
Therefore, $\psi$ is a random solution to random local optimization problem $RLOP(f, \RR^n)$. 
\end{proof}
\begin{corollary}\label{mslropcr1}
Let  $f: \Omega\times \RR^n\to \RR$ be a random function such that $f(\omega,x)$ is twice continuously differentiable in $x$ for each $\omega\in\Omega$. If  random variable $\xi:\Omega\to \RR^n$ such that $g(\omega,\xi(\omega))=0$ and $H(\omega,\xi(\omega))$ is positive definite for each $\omega\in\Omega$ then $\xi$ is a random solution to random local optimization problem $RLOP(f, \RR^n)$.
\end{corollary}
\begin{proof}
We use similar arguments as in the second part of the proof of Theorem \ref{mslrop}.
\end{proof}
\begin{corollary}\label{mslropcr2}
Let  $f: \Omega\times \RR^n\to \RR$ be a random function such that $f(\omega,x)$ is twice continuously differentiable and Hessian matrix $H(\omega, x)$ is positive semi-definite for any $x\in \RR^n$ and $\omega\in\Omega$. If  there exits a mapping $\xi:\Omega\to \RR^n$ such that $g(\omega,\xi(\omega))=0$ and $H(\omega,\xi(\omega))$ is positive definite for each $\omega\in\Omega$ then the random optimization problem $\underset{x\in \RR^n}{Inf} f(\omega, x)$ has a random solution.
\end{corollary}
\begin{proof}
For each $\omega\in\Omega$, by the positive semi-definiteness of Hessian matrix $H(\omega, x)$ for any $x\in\RR^n$, $f(\omega,x)$ is a convex function (\cite{T}, Proposition 2.2). We know that a local minimizer of a convex function on $\RR^n$ is also a global minimizer (\cite{T}, Proposition 2.30). So, by Theorem \ref{mslrop}, random solution to random local optimization problem $RLOP(f, \RR^n)$ is also a random solution tof random global optimization problem $ROP(f, \RR^n)$.
\end{proof}
Noting that, in Theorem \ref{mslrop}, not any mapping $\xi:\Omega\to \RR^n$ such that $g(\omega,\xi(\omega))=0$ and $H(\omega, \xi(\omega))$ is positive definite is a random solution of $RLOP(f, \RR^n)$. This fact is shown in the following example.
\begin{example}
Let $(\Omega, \mathcal F, \PP)$ be a complete probability space. We define mapping $f:\Omega\times\RR\to\RR$ by $f(\omega, x)=x^4-2x^2$ for any $\omega\in\Omega, x\in\RR.$ 
\end{example}
It is clear that $f(\omega, x)$ is a random function, and for each $\omega\in\Omega$, $f(\omega, x)$ is twice continuously differentiable. We have $g(\omega, x)=4x(x^2-1)$ and $H(\omega, x)=4(3x^2-1).$ Define $\xi_1, \xi_2:\Omega\to \RR$ by $\xi_1(\omega)=1, \xi_2(\omega)=-1$ for all $\omega\in\Omega$. Because $\xi_i$ is a random variable and $g(\omega,\xi_i(\omega))=0$, $H(\omega,\xi_i(\omega))=8$ is positive definite for any $\omega\in\Omega$, $i=1, 2$. So, $\xi_1, \xi_2$ are random solutions to random local optimization problem $RLOP(f, \RR)$. \par
Let $D$ is a non measurable subset of $\Omega$. We define the mapping $\xi:\Omega\to\RR$ by
$$\xi(\omega)=
\begin{cases}
\ \ 1& \text{if}\ \  \omega\in D\\
-1&\text{if}\ \  \omega\in \overline{D}.
\end{cases}$$
It is clear that $g(\omega,\xi(\omega))=0$ and $H(\omega,\xi(\omega))=8$ is positive definite for any $\omega\in\Omega.$ However, $\xi$ is not a random solution to random local optimization problem $RLOP(f, \RR)$ because it is not a random variable.

\end{document}